\documentclass[12pt,twoside,reqno]{amsart}
\usepackage{amsfonts}
\usepackage{amssymb}
\usepackage{amsmath}
\usepackage{bbm}
\usepackage[notcite,notref]{
}
\usepackage{natbib}

\usepackage[margin=0.9in]{geometry}
\usepackage{graphicx}
\usepackage{pgf,tikz}
\usepackage{float}
\usepackage{subcaption}
\usepackage{color}
\usepackage{enumitem}
\usepackage{float}

\theoremstyle{plain}
\newtheorem{theorem}{Theorem}[section]

\newtheorem{lemma}[theorem]{Lemma}
\newtheorem{proposition}[theorem]{Proposition}

\theoremstyle{remark}
\newtheorem{remark}{Remark}
\newcommand{\subscript}[2]{$#1 _ #2$}

\theoremstyle{definition}

\newcommand{\embed}{\hookrightarrow}
\newcommand{\embedembed}{\hookrightarrow\hookrightarrow}

\newcommand\A{\mathcal{A}}
\newcommand\R{\mathbb{R}}
\renewcommand\L{\mathcal{ L}}
\newcommand\e{\varepsilon}

\begin{document}


\title[Well-posedness and inverse Robin estimates]{Well-posedness and inverse Robin estimate for a multiscale elliptic/parabolic system}

\begin{abstract}
We establish the well-posedness of a coupled micro-macro parabolic-elliptic system modeling the interplay between two pressures in a gas-liquid mixture close to equilibrium that is filling a porous media with distributed microstructures. Additionally, we prove a local stability estimate for the inverse micro-macro Robin problem, potentially useful in identifying quantitatively a micro-macro interfacial Robin transfer coefficient given microscopic measurements on accessible fixed interfaces. To tackle the solvability issue we use two-scale energy estimates and two-scale regularity/compactness arguments cast in the Schauder's fixed point theorem. A number of auxiliary problems, regularity, and scaling arguments are used in ensuring the suitable Fr\'echet differentiability of the solution and the structure of the inverse stability estimate. 

\end{abstract}

\author{Martin Lind}
\address{Department of Mathematics and Computer Science\\Karlstad University\\
651 88 Karlstad\\ Sweden }
\email{martin.lind@kau.se}

\author{Adrian Muntean}
\address{Department of Mathematics and Computer Science\\Karlstad University\\
651 88 Karlstad\\ Sweden }
\email{adrian.muntean@kau.se}

\author{Omar Richardson}
\address{Department of Mathematics and Computer Science\\Karlstad University\\
651 88 Karlstad\\ Sweden }
\email{omar.richardson@kau.se}



\maketitle

\section{Introduction}

We are interested in developing evolution equations able to describe multiscale spatial interactions in gas-liquid mixtures, targeting a rigorous mathematical justification of Richards-like equations -  upscaled model equations generally chosen in a rather {\em ad hoc} manner by the engineering communities to describe the motion of flow in unsaturated porous media. The main issue is that  one lacks a rigorous derivation of the Darcy's law for such flow (see \cite{Hornung2012} (chapter 1) for a derivation via periodic homogenization techniques of the Darcy law for the saturated case).  

If air-water interfaces can be assume to be stagnant for a reasonable time span, then averaging techniques for materials with locally periodic microstructures (compare e.g. \cite{Chechkin1998}) lead in suitable scaling regimes to what we refer here as {\em two-pressure evolution systems}.  These are normally coupled parabolic-elliptic systems responsible for the joint evolution in time $t\in (0,T)$ ($T<+\infty$) of a parameter-dependent {\em microscopic pressure} $R\rho (t,x,y)$ evolving with respect to $y\in Y\subset \mathbb{R}^d$ for any given macroscopic spatial position $x\in \Omega$ and a {\em macroscopic pressure} $\pi(t,x)$ with $x\in \Omega$ for any given $t$. Here $R$ denotes the universal constant of gases. The two-scale geometry we have in mind is depicted in Figure 1 below.

\begin{figure}[H]
\centering
\includegraphics[scale=0.8]{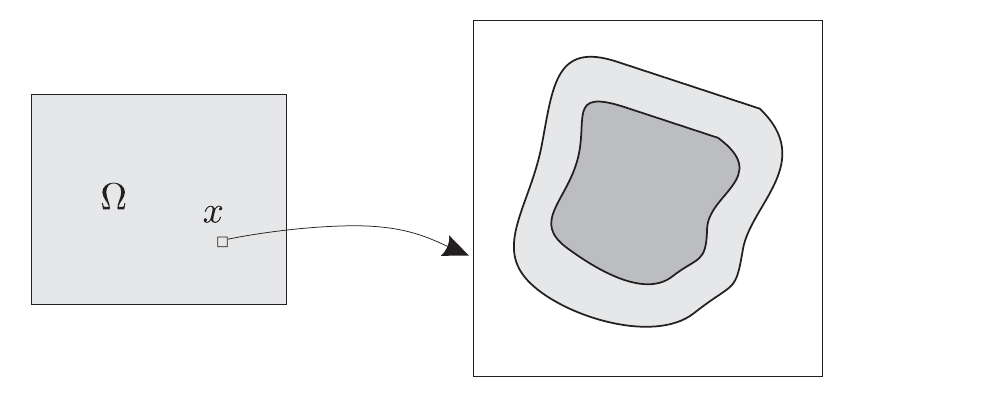}
\caption{The macroscopic domain $\Omega$ and microscopic pore $Y$ at $x\in\Omega$}
\end{figure} 

To cast the physical problem in mathematical terms as stated in \eqref{generalProblem}, we need a number of dimensional constant parameters ($A$ (gas permeability), $D$ (diffusion coefficient for the gaseous species),  $p_F$ (atmospheric pressure), $\rho_F$ (gas density))  and dimensional functions ($k$ (Robin coefficient) and $\rho_I$ (initial liquid density)). It is worth noting that excepting the Robin coefficient $k$, all the model parameters and functions are either known or can be accessed directly via measurements. Getting grip on a priori values of $k$ is more intricate simply because this coefficient is defined on the Robin part of the boundary of $\partial Y$, say $\Gamma_R$, where the micro-macro information transfer takes actively place. The Neumann part of the boundary  $\Gamma_N:=\partial Y-\Gamma_R$ is assumed to be accessible via measurements, while  $\Gamma_R$ is thought here as unaccessible. 

Our aim is twofold: 
\begin{itemize}
\item[(1)] ensure the well-posedness in a suitable sense of our two-pressure system with $k$ taken to be known;
\item[(2)] prove stability estimates with respect to $k$ for the inverse micro-macro Robin problem ($k$ is now unknown, but  measured values of the microscopic pressure are available on  $\Gamma_N$).
\end{itemize}
The main results reported here are Theorem \ref{existenceThm} (the weak solvability of \eqref{generalProblem}) and Theorem \ref{mainTeo} (the local stability for the inverse micro-macro Robin problem).

The choice of problem and approach is in line with other investigations running for two-scale systems, or systems with distributed microstructures, like \cite{feedback,sebamPhD,PesShow}.  As far as we are aware, this is for the first time that an inverse Robin problem is treated in a two-scale setting. A remotely connected single-scale inverse Robin problem is treated in \cite{nakamura15}.

\section{Problem formulation}


We shall consider the following parabolic-elliptic problem posed on two spatial scales $x\in\Omega$ and $y\in Y$.
\begin{equation}
\label{generalProblem}
\begin{cases}
-A\rho_F\Delta_x\pi=f(\pi,\rho) & \mbox{ in }\Omega\\
\partial_t\rho-D\Delta_y\rho = 0 & \mbox{ in }\Omega\times Y\\
-D\nabla_y\rho\cdot n_y=k(\pi+p_F-R\rho)&\mbox{ on } \Omega\times\Gamma_R\\
-D\nabla_y\rho\cdot n_y=0&\mbox{ on }\Omega\times\Gamma_N\\
\pi=0& \mbox{ at }\partial\Omega\\
\rho(t=0)=\rho_I(x,y)&\mbox{ in } \Omega\times Y,
\end{cases}
\end{equation}
where the parameters, coefficients and the nonlinear function $f$ satisfies the assumptions discussed below (see Section 2.1). The initial condition for $p$ follows from the coupling between $\rho$ and $p$. 

A prominent role in this paper is played by the micro-macro Robin transfer coefficient $k$, which is selected from the following set
\begin{equation}
\nonumber
\mathcal{K}:=\{k\in\L^2(\Gamma_R):0<\underline{k}\le k(y)\le\bar{k}\mbox{ for } y\in\Gamma_R\}. 
\end{equation}

\subsection{Assumptions}

\begin{enumerate}[label=(\subscript{A}{\arabic*})]
\item\label{domains} The domains $\Omega,Y$ have Lipschitz continuous boundaries.
\item\label{parameters} The parameters satisfy $A,D,\rho_F,R\in(0,\infty)$.
\item\label{initial} The initial value $\rho_I\in H^1(\Omega\times Y)$.
\item\label{scale} $f(\lambda u,\mu v)=\lambda^\alpha\mu^\beta f(u,v)$ where $\alpha+\beta=1$, $\alpha,\beta>0$.
\item\label{structuralAssump} 
There is a structural constant $C^*>0$ such that 
\begin{equation}
\nonumber
\int_{\Omega}|f(u_1,v)-f(u_2,v)|^2dx\le C^*\|u_1-v_2\|^2_{\L^2(\Omega)}
\end{equation}
uniformly in $v\in\L^2(\Omega;\L^2(Y))$.
\item\label{secondARG} There is a constant $C>0$ such that
$$
\int_\Omega f(u,v)^2dx\le C\|v\|^2_{\L^2(\Omega;H^1(Y))}.
$$
\item\label{structuralAssump2}
The constant $C^*$ in \ref{structuralAssump} satisfies
\begin{equation}
\nonumber
C^*c_P(\Omega)<1,
\end{equation}
where $c_P(\Omega)$ is the Poincar\'{e} constant of the domain $\Omega$ (see Proposition \ref{poincare} below). 
\end{enumerate}

\begin{remark}
Assumptions \ref{domains}-\ref{initial} have clear geometrical or physical meanings, while \ref{scale}-\ref{structuralAssump2} are technical. The assumption \ref{structuralAssump2} is only used when deriving uniqueness of the weak solution to (\ref{generalProblem}). Note also that for some special classes of domains, the Poincar\'{e} constant can be quantitatively estimated, see e.g. \cite{mikhlin81}.
\end{remark}

\subsection{Auxiliary results}

In this section, we state some auxiliary results that will be useful in this context.

\begin{proposition}[Poincar\'{e}'s inequality]
\label{poincare}
Let $\Omega\subset\R^d$ be a fixed domain and denote by $c_P(\Omega)$ the smallest constant such that
\begin{equation}
\nonumber
\|u\|^2_{\L^2(\Omega)}\le c_P(\Omega)\|\nabla_xu\|^2_{\L^2(\Omega)}
\end{equation}
hold for all $u\in H_0^1(\Omega)$. The constant $c_P(\Omega)$ is called the Poincar\'{e} constant of the domain $\Omega$.
\end{proposition}
\begin{proposition}[Interpolation-trace inequality]
\label{sub:trace_theorem}
Assume that $Y\subset\R^d$ is a Lipschitz domain and $u\in\L^2(\Omega;H^1(Y))$. For any $\rho>0$ we have
\begin{equation}
\nonumber
\int_\Omega\int_{\partial Y}u^2d\sigma_y dx\le\rho\int_\Omega\int_Y|\nabla_yu|^2dydx+c_\rho\int_\Omega\int_Y|u|^2dydx,
\end{equation}
where $c_\rho\sim1/\rho$. In particular,
\begin{equation}
\nonumber
\|u\|_{\L^2(\Omega;\L^2(\partial Y))}\le c\|u\|_{\L^2(\Omega;H^1(Y))}.
\end{equation}
\end{proposition}
The next result provides a useful equivalent norm on $H^1$.
\begin{proposition}
\label{equivalenSob}
Let $U\subset\R^d$ be a domain and $\Gamma\subset\partial Y$ where $\Gamma$ has positive $(d-1)$-dimensional surface measure. Then there are constants $c_1,c_2$ such that
\begin{equation}
\nonumber
c_1\|u\|^2_{H^1(U)}\le\int_\Gamma u^2d\sigma+\|\nabla_xu\|^2_{\L^2(\Omega)}\le c_2\|u\|^2_{H^1(U)}.
\end{equation}
\end{proposition}
We shall also need the following two results, the first a Sobolev-type embedding and the second a simple trace theorem.
\begin{proposition}
\label{sobolev}
Assume that $U\subset\R^d$, then $H^{d}(U)\subset\L^\infty(U)$ and
\begin{equation}
\nonumber
\|u\|_{\L^\infty(U)}\le C\|u\|_{H^d(U)}.
\end{equation}
\end{proposition}
\begin{proposition}
\label{trace}
Assume that $U\subset\R^d$ and $\Gamma\subset\partial U$ is Lipschitz continuous. Then 
$$
H^{d+1}(U)\subset H^d(\Gamma)
$$
and
$$
\|u\|_{H^d(\Gamma)}\le c\|u\|_{H^{d+1}(U)}.
$$
\end{proposition}

We have the following existence and regularity results.
\begin{proposition}[see e.g. \cite{evans98}]
Consider the problem
\label{parabolic}
\begin{equation}
\label{paraboliceq}
\begin{cases}
\partial_tv - D\Delta_y v = 0 & \mbox{ on } \Omega\times Y\\
-D\nabla_yv\cdot n_y=k(g - Rv) & \mbox{ on } \Omega\times\Gamma_R\\
-D\nabla_yv\cdot n_y= 0&\mbox{ on }\Omega\times\Gamma_N\\
v(t=0)=v_I&\mbox{ on }\Omega\times Y.
\end{cases}
\end{equation}
If $g\in\L^2(0,T;H^1(\Omega))$ and $v_I\in H^1(\Omega;H^m(Y))~(m\in\mathbb{N})$, then
the problem (\ref{paraboliceq}) has a unique weak solution $v\in \L^2(0,T;H^1(\Omega;H^{m+1}(Y)))$.
\end{proposition}
\begin{proposition}[see e.g \cite{cazenave06}]
Let $v(t,x,y)$ be in $\L^2(0,T;H^1(\Omega;H^{m+1}(Y)))~(m\in\mathbb{N})$ and consider the problem
\label{elliptic}
\begin{equation}
\label{ellipticeq}
\begin{cases}
-\Delta_x u=f(u,v) & \mbox{ on } \Omega\times Y\\
u(t,x)=0&\mbox{ on }\partial\Omega, t>0.
\end{cases}
\end{equation}
where the nonlinear function satisfies \ref{scale}-\ref{secondARG}.
Then the problem (\ref{ellipticeq}) has a unique weak solution $u\in\L^2(0,T;H^1(\Omega))$.
\end{proposition}

Finally, we state the following two classical compactness results, see e.g. \cite{zeidler86}.
\begin{theorem}[Aubin-Lions Theorem \cite{aubin1963}]
\label{sub:aubin_lions}
Let $B_0 \embed B \embedembed B_1$. Suppose that $B_0$ is compactly embedded in $B$ and that $B$ is continuously embedded in $B_1$. Let
\begin{equation}
\label{eq:aubin_lions}
W = \left\{u\in \L^2\left(0,T;B_0 \right) : \partial_t u \in \L^2\left(0,T; B_1 \right) \right\}.
\end{equation}
Then the embedding of $W$ into $\L^2\left(0,T;B \right)$ is compact.
\end{theorem}

\begin{theorem}[Schauder's Fixed Point Theorem]
\label{sub:schauders_fixed_point_theorem}
Let $B$ be a nonempty, closed, convex, bounded set and $T:B\to B$ a compact operator. Then there exists at least one $r \in B$ such that $T(r) = r$. 
\end{theorem}

\section{Existence and uniquenes of the solution}
\subsection{Existence of weak solution}
The main result of this subsection is the following theorem.
\begin{theorem}
\label{existenceThm}
Assume that \ref{parameters}-\ref{secondARG} hold. Then the problem (\ref{generalProblem}) has at least a weak solution solution $(\pi,\rho)\in\L^2(0,T;H_0^1(\Omega))\times\L^2(0,T;\L^2(\Omega;H^1(Y)))$. 
\end{theorem}
\begin{proof}
We shall decouple the problem. The first sub-problem is as follows:
given $\pi\in\L^2(0,T;H^1_0(\Omega))$ and $\rho_I\in H^1(\Omega,H^1(Y))$, we let $\xi$ be the weak solution to
\begin{equation}
\label{eq:p2}
\begin{cases}
\partial_t\xi - D\Delta_y \xi = 0 & \mbox{ on } \Omega\times Y\\
-D\nabla_y\xi\cdot n_y=k(\pi+p_F - R\xi) & \mbox{ on } \Omega\times\Gamma_R\\
-D\nabla_y\xi\cdot n_y= 0&\mbox{ on }\Omega\times\Gamma_N\\
\xi(t=0)=\rho_I&\mbox{ on }\Omega\times Y.
\end{cases}
\end{equation}
The weak formulation of (\ref{eq:p2}) is: find $\xi$ such that for a.e. $t\in[0,T]$ and every $\psi\in\L^2(\Omega, H^1(Y))$ there holds
\begin{equation}
\label{eq:weak_p2}
\int_\Omega\int_Y \partial_t\xi\psi dydx+\int_\Omega\int_Y D\nabla_y \xi\nabla_y \psi dydx=\int_\Omega\int_{\Gamma_R}k (\pi+p_F-R\xi)\psi d\sigma_ydx,  
\end{equation}
and $\xi(t=0)=\rho_I$.
Existence and regularity of $\xi$ is provided by Proposition \ref{parabolic} (recall that \ref{initial} states that $\rho_I\in H^1(\Omega\times Y)$).

The second sub-problem is: given data $\xi$, consider the problem
\begin{equation}
\label{eq:p3}
\begin{cases}
-\Delta_x \pi=f(\pi,\xi) &\mbox{ on }\Omega\\
\pi=0 &\mbox{ on }\partial\Omega.
\end{cases} 
\end{equation}
Let $\lambda>0$ be a free parameter. By the scaling properties of $f$ and uniqueness of weak solution, we have that if $\pi$ is the weak solution of (\ref{eq:p3}) with data $\xi$, then $\bar{\pi}=\lambda\pi$ is the weak solution to (\ref{eq:p3}) with data $\lambda\xi$. Hence, if $\bar{\pi}$ is the weak solution to
\begin{equation}
\label{eq:p3lambda}
\begin{cases}
-\Delta_x\bar{\pi}=\lambda^\beta f(\bar{\pi},\xi) &\mbox{ on }\Omega\\
\bar{\pi}=0 &\mbox{ on }\partial\Omega,
\end{cases} 
\end{equation}
then $\bar{\pi}=\lambda\pi$, again by the scaling properties of $f$.
The weak form of (\ref{eq:p3lambda}) is as follows: find $\bar{\pi}$ such that for a.e. $t\in[0,T]$ and all $\varphi\in H_0^1(\Omega)$, there holds
\begin{equation}
\nonumber
\label{eq:weak_pressure}
\int_\Omega \nabla_x\pi\cdot\nabla_x\varphi dx=\lambda^\beta\int_\Omega f(\pi,\xi)\varphi dx. 
\end{equation}
Existence and regularity of $\bar{\pi}$ is guaranteed by Proposition \ref{elliptic}.

We shall now use a fixed point argument \`{a} la Schauder (see Theorem \ref{sub:schauders_fixed_point_theorem}) to show that there exists a $\lambda>0$ for which the functions of the pair $(\bar{\pi},\xi)$  are weak solutions to the sub-problems (\ref{eq:p2}) and (\ref{eq:p3lambda}). Then we recover $(\pi,\rho)$,  a weak solution to (\ref{generalProblem}), by taking $\pi=\bar{\pi}/\lambda$ and $\rho=\xi$.

Define the operators
$$
T_1: \L^2(0,T;\L^2(\Omega))\rightarrow \L^2(0,T; \L^2(\Omega;\L^2(Y)))
$$
by $T_1(\pi)=\xi$ (the weak solution of (\ref{eq:p2})) and
$$
T^\lambda_2:\L^2(0,T; \L^2(\Omega;H^1(Y)))\rightarrow \L^2(0,T; H^1(\Omega))
$$
by $T^\lambda_2(\xi)=\bar{\pi}$ (the weak solution of (\ref{eq:p3lambda})). Finally, consider the operator $\mathcal{A}^\lambda$ on the space $\L^2(0,T;\L^2(\Omega))$ into itself defined by
\begin{equation}
\label{eq:operatorA}
\mathcal{A}^\lambda(\pi)=T^\lambda_2(T_1(\pi))=\bar{\pi}.
\end{equation}
To obtain existence of solution, we shall prove that the operator $\mathcal{A}^\lambda$ has a fixed point. This $\pi$ will then give $\xi$.
The idea of the proof is to first use the Schauder Fixed Point Theorem (Theorem \ref{sub:schauders_fixed_point_theorem} above).

We shall prove that there exist a $\lambda>0$ and a set $B$ such that
\begin{enumerate}
\item\label{prop1} $\A^\lambda$ is a compact operator;
\item\label{prop2} $B$ is convex, closed, bounded and satisfies 
$\A^\lambda(B)\subset B$.
\end{enumerate}
To obtain compactness of $A^\lambda=T^\lambda_2\circ T_1$, it is sufficient to demonstrate that $T_1$ is compact and that $T^\lambda_2$ is continuous. Recall that we have
$$
T_1:\L^2(0,T;H^1(\Omega))\rightarrow \L^2(0,T; \L^2(\Omega;\L^2(Y))).
$$
However, since we assume that $\xi_I\in H^1(\Omega, H^1(Y))$ we get
that $T_1(\pi)=\xi\in\L^2(0,T;H^1(\Omega\times Y))$ and $\partial_t\xi\in\L^2(0,T;L^2(\Omega\times Y))$.
Whence,
$$
T_1( \L^2(0,T;H^1(\Omega)))\subset V,
$$
where
$$
V=\{u:u\in\L^2(0,T;H^1(\Omega\times Y)),\, \partial_tu\in\L^2(0,T;L^2(\Omega\times Y))\}
$$
By Theorem \ref{sub:aubin_lions},
$$
V\embedembed \L^2(0,T;\L^2(\Omega;\L^2(Y))).
$$
Thus, for any bounded set $M\subset\L^2(0,T; \L^2(\Omega;\L^2(Y)))\times \L^2(0,T;\L^2(\Omega))$, there holds $T_1(M)\subset V$ and since $V$ is compactly contained in $\L^2(0,T;\L^2(\Omega;\L^2(Y))$ we have that $T_1(M)$ is precompact in $\L^2(0,T;\L^2(\Omega;\L^2(Y)))$. Hence, $T_1$ is compact.

We continue to prove that $T^\lambda_2$ is continuous.
Assume we have two solutions $\bar{\pi}_1=T_2^\lambda(\xi_1)$ and $\bar{\pi}_2=T^\lambda_2(\xi_2)$.
Substituting these both in \eqref{eq:weak_pressure} and subtracting, we obtain
\begin{equation}
\nonumber
\int_\Omega\nabla_x(\bar{\pi}_1-\bar{\pi}_2)\cdot\nabla_x\varphi dx =\lambda^\beta\int_\Omega[f(\bar{\pi}_1,\xi_1) - f(\bar{\pi}_2,\xi_2)]\varphi dx,
\end{equation}
and for $\varphi=\bar{\pi}_1-\bar{\pi}_2$, we get
\begin{align*}
\|\nabla_x(\bar{\pi}_1-\bar{\pi}_2)\|_{\L^2(\Omega)}^2
&=\lambda^\beta\int_\Omega[f(\bar{\pi}_1,\xi_1)-f(\bar{\pi}_2,\xi_2)][\bar{\pi}_1-\bar{\pi}_2]dx\\
&=\lambda^\beta\int_\Omega[f(\bar{\pi}_1,\xi_1)-f(\bar{\pi}_2,\xi_1)][\bar{\pi}_1-\bar{\pi}_2]dx\\
&+\lambda^\beta\int_\Omega[f(\bar{\pi}_2,\xi_1) - f(\bar{\pi}_2,\xi_2)][\bar{\pi}_1-\bar{\pi}_2]dx.
\end{align*}
Using \ref{structuralAssump} and \ref{secondARG}, we obtain that
\begin{equation}
\nonumber
\|\nabla_x(\bar{\pi}_1-\bar{\pi}_2)\|_{\L^2(\Omega)}^2\le C\lambda^\beta\|\bar{\pi}_1-\bar{\pi}_2\|_{\L^2(\Omega)}\|\xi_2-\xi_1\|_{\L^2(\Omega;H^1(Y))}.
\end{equation}
By the Poincar\'e's inequality, we obtain
\begin{equation}
\nonumber
\|\nabla_x(\bar{\pi}_1-\bar{\pi}_2)\|_{\L^2(\Omega)}\leq C\lambda^\beta\|\xi_1-\xi_2\|_{\L^2(\Omega;H^1(Y))}   
\end{equation}
and we conclude the mapping $T^\lambda_2$ is continuous.

Let $K>0$ be a fixed number that we specify later and let $B_K$ be the collection of functions $u\in\L^2(0,T;H^1(\Omega))$ such that
\begin{equation}
\nonumber
\max\{\|u\|_{\L^2(0,T;L^2(\Omega))},\|\nabla_xu\|_{\L^2(0,T;L^2(\Omega))} \}\le K.
\end{equation}
For each $K>0$, the set
$$
B_K\subset \L^2(0,T;H^1(\Omega))
$$
is a convex, closed and bounded. We show that we may select $K>0$ and $\lambda>0$ such that
\begin{equation}
\label{eq:inclusion}
\A^\lambda(B_K)\subset B_K.
\end{equation}
Note that $T_1(B_K)$ is a bounded subset of $\L^2(0,T;\L^2(\Omega;\L^2(Y)))$, with a bound depending only on $K$. In other words,
\begin{equation}
\label{eq:xi_bound}
\|\xi\|_{\L^2(0,T;\L^2(\Omega;\L^2(Y))))}\le CK.
\end{equation}
Indeed, this follows from the fact that $T_1$ is a compact operator.

We proceed by observing that we may choose $\lambda>0$ such that if $u\in B_K$ is arbitrary and $v=T_2^\lambda(T_1(u))$, then
\begin{equation}
\label{eq:maxEst2}
\max\left\{\|v\|_{\L^2(0,T;\L^2(\Omega))},\|v\|_{\L^2(0,T;H^1(\Omega))}\right\}\le K.
\end{equation}
Let $\xi=T_1(u)$ so that $v=T_2^\lambda(\xi)$. Testing the weak formulation of $(\ref{eq:p3lambda})$ with $\varphi=u$ and using Cauchy-Schwarz' inequality and Poincar\'e's inequality, we get
$$
\|\nabla_xu\|_{\L^2(\Omega)}^2\le C\lambda^\beta\|u\|_{\L^2(\Omega)}\|\xi\|_{\L^2(\Omega;\L^2(Y))}\le CK\lambda^\beta\|\xi\|_{\L^2(\Omega;\L^2(Y))}.
$$
Integrating over $[0,T]$ and using (\ref{eq:xi_bound}),
we obtain after using Poincar\'e's inequality
$$
\|u\|^2_{\L^2(0,T;H^1(\Omega))}\le C'K^2\lambda^\beta.
$$
By taking $\lambda$ small enough (depending on $K$), we obtain (\ref{eq:maxEst2}) whence (\ref{eq:inclusion}) follows.
\end{proof}

\begin{remark}
Instead of using scaling arguments and Schauder's fixed point theorem, we could have used alternatively the Schaefer/Leray-Schauder fixed point theorem.
\end{remark}

\subsection{Uniqueness of weak solutions}

We proceed to prove the following uniqueness theorem.
\begin{theorem}
\label{uniquenessThm}
Assume that in addition to the assumptions of Theorem \ref{existenceThm} the condition \ref{structuralAssump2} also holds. Then the weak solution to (\ref{generalProblem}) is unique.
\end{theorem}
\begin{proof}
The weak formulation of the uncoupled problem is:
find $(\pi,\rho)\in\L^2(0,T;H_0^1(\Omega))\times \L^2(0,T;\L^2(\Omega;H^1(Y)))$ where $\rho(t=0)=\rho_I$ and for a.e. $t\in[0,T]$ the equations
\begin{equation}
\label{eq:p2_uniq}
A\rho_F\int_\Omega\nabla_x\pi\cdot\nabla_x\varphi dx=\int_\Omega f(\pi,\rho)\varphi dx
\end{equation}
and
\begin{equation}
\label{eq:p3_uniq}
\int_\Omega\int_Y\partial_t\rho\psi dydx+\int_\Omega \int_YD\nabla_y \rho\cdot\nabla_y\psi dydx= \int_\Omega\int_{\Gamma_R}k(\pi+p_F-R\rho)\psi d\sigma_ydx  
\end{equation}
hold for all $\varphi\in H^1_0(\Omega)$ and all $\psi\in\L^2(\Omega;H^1(Y))$.

Assume that two pairs of solutions exist: $(\pi_1,\rho_1)$ and $(\pi_2,\rho_2)$. Let $q:=\pi_1-\pi_2$ and $z:=\rho_1-\rho_2$.
If we substitute the two solutions in \eqref{eq:p2_uniq} and \eqref{eq:p3_uniq} and subtract, we obtain that
\begin{equation}
\label{eq:p2_uniq_pair}
A\rho_F \int_\Omega\nabla_x q\cdot\nabla_x\varphi dx=\int_\Omega \left( f(\pi_1,\rho_1) -f(\pi_2,\rho_2) \right)\varphi dx
\end{equation}
and
\begin{equation}
\label{eq:p3_uniq_pair}
\int_\Omega\int_Y\partial_t z \psi dydx+\int_\Omega\int_YD\nabla_y z\cdot\nabla_y\psi dydx=\int_\Omega\int_{\Gamma_R} k(q - Rz)\psi d\sigma_ydx
\end{equation}
for all $\varphi\in H^1_0(\Omega)$ and all $\psi\in\L^2(\Omega;H^1(Y))$.

Choosing specific test function $\varphi = q$, using Young's inequality with parameter $\e_1>0$ and \ref{structuralAssump}, we obtain from \eqref{eq:p2_uniq_pair} the first key estimate
\begin{align}
\label{eq:key_estimate1}
A\rho_F \|\nabla_x q\|^2_{\L^2(\Omega)}&\leq (C^*+\e_1)\|q\|_{\L^2(\Omega)}^2+c_{\e_1}\|z\|^2_{\L^2(\Omega,\L^2(Y))}.
\end{align}
We focus on \eqref{eq:p3_uniq_pair}, which, using test function $\psi=z$, yields
\begin{equation}
\label{eq:energy}
\frac{1}{2}\frac{d}{dt}||z||^2_{\L^2(\Omega;\L^2(Y))} + D||\nabla_y z||^2_{\L^2(\Omega;\L^2(Y))}=\int_\Omega \int_{\Gamma_R} k(q - Rz)z.
\end{equation}
Now, we estimate the right hand side of (\ref{eq:energy}) by using trace inequality and the fact that $k\le\bar{k}$ on $\Gamma_R$. We have
\begin{eqnarray}
\nonumber
\int_\Omega\int_{\Gamma_R}|k(q-Rz)z|d\sigma_y dx&\le&\bar{k}\int_\Omega\int_{\Gamma_R}|qz|d\sigma_y dx+R\bar{k}\int_\Omega\int_{\Gamma_R}z^2d\sigma_y dx\\
\nonumber
&\le&\frac{\bar{k}|\Gamma_R|}{2}\|q\|^2_{\L^2(\Omega)}+\left(R+\frac{1}{2}\right)\bar{k}\int_\Omega\int_{\Gamma_R}z^2d\sigma_y dx.
\end{eqnarray}
The second term at the right-hand side of the previous inequality can be estimated by using the trace inequality and Young's inequality with parameter $\e>0$:
\begin{eqnarray}
\nonumber
\int_\Omega\int_{\Gamma_R}z^2d\sigma_y dx\le\e\|\nabla_yz\|^2_{\L^2(\Omega;\L^2(Y))}+\frac{c_0}{\e}\|z\|^2_{\L^2(\Omega;\L^2(Y))}
\end{eqnarray}
for some absolute constant $c_0>0$.
Using the previous estimates and rearranging \eqref{eq:energy}, we obtain
\begin{eqnarray}
\nonumber
\frac{1}{2}\frac{d}{dt}||z||^2_{\L^2(\Omega;\L^2(Y))} + (D-\e)||\nabla_y z||^2_{\L^2(\Omega;\L^2(Y))}\\
\label{epsilonest}
\le\frac{\bar{k}|\Gamma_R|}{2}\|q\|^2_{\L^2(\Omega)}+\frac{c_0\bar{k}(R+1/2)}{\e}\|z\|^2_{\L^2(\Omega;\L^2(Y))}.
\end{eqnarray}
 By Poincare's inequality, we have
\begin{equation}
\nonumber
\|q\|^2_{\L^2(\Omega)}\le c_P(\Omega)\|\nabla_x q\|^2_{\L^2(\Omega)},
\end{equation}
where $c_P(\Omega)$ is the Poincar\'{e} constant of the domain $\Omega$. Using this in \eqref{eq:key_estimate1}, we obtain
\begin{equation}
\nonumber
A\rho_F\|\nabla_xq\|^2_{\L^2(\Omega)}\le (C^*+\e_1)c_P(\Omega)\|\nabla_xq\|^2_{\L^2(\Omega)}+ C\|z\|^2_{\L^2(\Omega,\L^2(Y))}.
\end{equation}
By \ref{structuralAssump2}, we may take $\e_1>0$ small enough such that $(C^*+\e_1)c_P(\Omega)<A\rho_F$. Then we obtain
\begin{equation}
\label{structEst}
\|q\|^2_{\L^2(\Omega)}\le c_P(\Omega)\|\nabla_xq\|^2_{\L^2(\Omega)}\le \frac{Cc_P(\Omega)}{A\rho_F-(C^*+\e_1)c_P(\Omega)}\|z\|^2_{\L^2(\Omega,\L^2(Y))}.
\end{equation}
Whence, it follows from the previous estimate and (\ref{epsilonest}) with $\e=D/2$ that
\begin{equation}
\nonumber
\frac{1}{2}\frac{d}{dt}||z||^2_{\L^2(\Omega;\L^2(Y))} + \frac{D}{2}||\nabla_y z||^2_{\L^2(\Omega;\L^2(Y))}\le C\|z\|^2_{\L^2(\Omega;\L^2(Y))}
\end{equation}
By using Gr\"{o}nwall's inequality and the fact that $z(0,x,y)=0$, it follows that $z=0$. From (\ref{structEst}), we obtain $q=0$ as well. This demonstrates the uniqueness.
\end{proof}

\section{Energy and stability estimates}

We start this section by stating the following energy estimates for our problem.

\begin{proposition}
\label{energy1}
Assume \ref{parameters}-\ref{secondARG} and let $(u,v)$ be a weak solution to
\begin{equation}
\label{en1}
\begin{cases}
-\Delta_xu=f(u,v) & \mbox{ in }\Omega\\
\partial_tv-D\Delta_yv = 0 & \mbox{ in }\Omega\times Y\\
-D\nabla_yv\cdot n_y+k(Rv-u)=g&\mbox{ on } \Omega\times\Gamma_R\\
-D\nabla_yv\cdot n_y=0&\mbox{ on }\Omega\times\Gamma_N\\
u=0& \mbox{ at }\partial\Omega\\
v(t=0)=v_I&\mbox{ in } \Omega\times Y.
\end{cases}
\end{equation}
Then the following energy estimate hold
\begin{align}
\nonumber
&\|u\|^2_{\L^2(0,T;H^1(\Omega))}+\|v\|^2_{\L^2(0,T;\L^2(\Omega,H^1(Y)))}\\
\label{en2}
&\le C\left(\|g\|^2_{\L^2(0,T;\L^2(\Omega;\L^2(\Gamma_R)))}+\|v_I\|^2_{\L^2(\Omega;\L^2(Y))}\right).
\end{align}
\end{proposition}
The proof of Proposition \ref{energy1} follows by similar arguments as the proof of Theorem \ref{stability1} below, therefore we omit it.

We proceed to study the stability of solutions with respect to some of the parameters involved. Some preliminary remarks:
\begin{itemize}
\item We do not need to study the stability of the solution with respect to $\rho_F, p_F$ and $R$. Recall that $R$ is an universal physical constant, while $\rho_F, p_F$ fix the type of fluid and gas we are considering.
\item We could investigate the stability of $(\pi,\rho)$ with respect to structural changes into the non-linearity $f(\cdot,\cdot)$.  We omit to do so mainly because our main intent lies in understanding the role of the micro-macro Robin coefficient $k$.
\item For this stability proof, we decide to use a direct method which relies essentially on energy estimates; see e.g. \cite{munty2009}.
\end{itemize}

For $i\in\{ 1,2\}$, let $(\pi_i,\rho_i)$ be two weak solutions  corresponding to the sets of data $(\rho_{Ii},A_i, D_i, k_i)$,
where $\rho_{Ii},A_i, D_i, k_i$ denote the initial data, diffusion coefficients and mass-transfer coefficients of the solution $(\pi_i,\rho_i)$. Denote
\begin{equation}
\nonumber
\delta u:=u_2-u_1\quad\mbox{where}\quad u\in\left\{\pi,\rho,\rho_I, A, D, k\right\}.
\end{equation}

\begin{theorem}
\label{stability1}
Assume that for $i=1,2$, $(A_i,D_i)$ belongs to a fixed compact subset of $\R^2$, that $k_i\in\mathcal{K}$ and that $\|\rho_{Ii}\|_{\L^2(\Omega;\L^2(Y))}\le C$. 
Let $(\pi_i,\rho_i)~~(i=1,2)$ be weak solutions to (\ref{generalProblem}) corresponding to the choices of data above. Then the estimate
\begin{align}
\nonumber
\|\delta\pi\|^2_{\L^2(0,T;H_0^1(\Omega))}+\|\delta\rho\|^2_{\L^2(0,T;\L^2(\Omega;H^1(Y)))}\\
\label{eq:cont_data_ineq}
\le c\left(\|\delta k\|^2_{\L^2(\Gamma_R)}+|\delta A|+|\delta D|+\|\delta\rho_I\|^2_{\L^2(\Omega;\L^2(Y))}\right),
\end{align}
holds
\end{theorem}


\begin{proof}
We have for $i=1,2$
\begin{equation}
\nonumber
A_i\rho_F\int_\Omega\nabla_x\pi_i\cdot\nabla_x\varphi dx=\int_\Omega f(\pi_i,\rho_i)\varphi dx
\end{equation}
and
\begin{equation}
\nonumber
\int_\Omega\int_Y\partial_t\rho_i\psi dydx+\int_\Omega\int_YD_i \nabla_y\rho_i\cdot\nabla_y\psi dydx=\int_\Omega\int_{\Gamma_R}k_i \left( \pi_i + p_F - R\rho_i\right)\psi d\sigma_ydx,
\end{equation}
for all $\varphi \in H^1_0(\Omega)$ and $\psi\in\L^2(\Omega;H^1(Y))$.

Subtracting the corresponding equations and then testing with $\varphi:=\pi_2-\pi_1$ and $\psi:=\rho_2-\rho_1$ gives:
\begin{equation}
\label{eq:alpha}
\rho_F\left(A_2\int_\Omega\nabla_x\pi_2\cdot\nabla_x\varphi dx-A_1\int_\Omega\nabla_x\pi_1\cdot\nabla_x\varphi dx\right) = \int_\Omega\left( f(\pi_2,\rho_2)- f(\pi_1,\rho_1)\right)\varphi\,dx,   
\end{equation}
and
\begin{equation}
\begin{split}
&\frac{d}{2dt}\|\psi\|^2_{\L^2(\Omega;\L^2(Y))}+\int_\Omega\int_Y\left( D_2 \nabla_y \rho_2 - D_1 \nabla_y \rho_1\right)\cdot\nabla_y\psi dydx\\
\label{eq:beta}
&=\int_\Omega\int_{\Gamma_R}\left(k_2(\pi_2+p_F-R\rho_2)-k_1(\pi_1+p_F -R\rho_1)\right)\psi d\sigma_ydx.
\end{split}
\end{equation}
Regarding \eqref{eq:alpha}, note that
\begin{eqnarray}
\nonumber
A_2\int_\Omega\nabla_x\pi_2\cdot\nabla_x\varphi dx-A_1\int_\Omega\nabla_x\pi_1\cdot\nabla_x\varphi dx\\
\nonumber
=A_2\|\nabla_x\varphi\|^2_{\L^2(\Omega)}+(A_2-A_1)\int\nabla_x\pi_1\cdot\nabla_x\varphi dx
\end{eqnarray} 
Using \ref{structuralAssump} and \ref{secondARG}, we may estimate the right-hand side of (\ref{eq:alpha}) and obtain
\begin{equation}
\nonumber
A_2\rho_F\|\nabla_x\varphi\|^2_{\L^2(\Omega)}\le C^*\|\varphi\|^2_{\L^2(\Omega)}+ c\left(\|\psi\|^2_{\L^2(\Omega;\L^2(Y))}+|\delta A|\int_\Omega|\nabla_x\pi_1||\nabla_x\varphi|dx\right).
\end{equation}
Using Poincar\'{e}'s inequality, assumptions on $f$ and Young's inequality with parameter $\e>0$, we get
\begin{align}
\nonumber
\|\nabla_x\varphi\|^2_{\L^2(\Omega)}&\le  c\left(\|\psi\|^2_{\L^2(\Omega;\L^2(Y))}+|\delta A|\int_\Omega|\nabla_x\pi_1||\nabla_x\varphi|dx\right)\\
\nonumber
&\le c\e\|\nabla_x\varphi\|^2_{\L^2(\Omega)}+c\left(\|\psi\|^2_{\L^2(\Omega;\L^2(Y))}+|\delta A|\|\nabla_x\pi_1\|^2_{\L^2(\Omega)}\right).
\end{align}
Choosing $\e=1/(2c)$, rearranging and using energy estimates for $\pi_1$, we obtain
\begin{equation}
\label{phiestimate}
\|\nabla_x\varphi\|^2_{\L^2(\Omega)}\le c\left(\|\psi\|^2_{\L^2(\Omega;\L^2(Y))}+|\delta A|\right).
\end{equation}
We proceed to estimate $\|\psi\|^2_{\L^2(\Omega;\L^2(Y))}$, using \eqref{eq:beta}. Note that
\begin{align}
\nonumber
&\int_\Omega\int_Y\left( D_2 \nabla_y\rho_2 - D_1 \nabla_y \rho_1\right)\cdot\nabla_y\psi dydx\\
\nonumber
&=(D_2-D_1)\int_\Omega\int_Y\nabla_y\rho_2\cdot\nabla_y\psi+D_1\|\nabla_y\psi\|^2_{\L^2(\Omega;\L^2(Y))}.
\end{align}
Hence, it follows that
\begin{align}
\nonumber
&\frac{d}{2dt}\|\psi\|^2_{\L^2(\Omega;\L^2(Y))}+D_1\|\nabla_y\psi\|^2_{\L^2(\Omega;\L^2(Y))}\le|\delta D|\int_\Omega\int_Y|\nabla_y\xi_1||\nabla_y\psi\|dydx\\
&+\int_\Omega\int_{\Gamma_R}|(k_2-k_1)\psi|+|(k_2\pi_2 - k_1\pi_1)\psi|+|R(k_2\rho_2 - k_1\rho_1)\psi| d\sigma_ydx.
\end{align}
We have
\begin{eqnarray}
\nonumber
\int_\Omega\int_{\Gamma_R}|k_2-k_1||\psi|d\sigma_ydx&\le&\e\int_\Omega\int_{\Gamma_R}\psi^2d\sigma_ydx+c_\e|\Omega|\|k_2-k_1\|^2_{\L^2(\Gamma_R)}\\
\nonumber
&\le&c\e\left(\|\psi\|^2_{\L^2(\Omega;\L^2(Y))}+\|\nabla_y\psi\|^2_{\L^2(\Omega;\L^2(Y))}\right)+c\|k_2-k_1\|^2_{\L^2(\Gamma_R)}.
\end{eqnarray}
Further,
\begin{align}
\nonumber
&\int_\Omega\int_{\Gamma_R}|(k_2\pi_2 - k_1\pi_1)\psi|d\sigma_ydx\le c\|k_2-k_1\|^2_{\L^2(\Gamma_R)}+\int_\Omega\int_{\Gamma_R}|k_1||\pi_2-\pi_1||\psi|d\sigma_ydx\\
\nonumber
&\le c\|k_2-k_1\|^2_{\L^2(\Gamma_R)}+\e\bar{k}|\Gamma_R|\|\pi_2-\pi_1\|^2_{\L^2(\Omega)}+c_\e\int_\Omega\int_{\Gamma_R}\psi^2d\sigma_ydx\\
\nonumber
&\le c\|k_2-k_1\|^2_{\L^2(\Gamma_R)}+\e\left(\|\nabla_x\varphi\|^2_{\L^2(\Omega)}+\|\nabla_y\psi\|^2_{\L^2(\Omega;\L^2(Y))}\right)+c_\e\|\psi\|^2_{\L^2(\Omega;\L^2(Y))}.
\end{align}
Finally,
\begin{align}
\nonumber
&R\int_\Omega\int_{\Gamma_R}|(k_2\rho_2 - k_1\rho_1)\psi| d\sigma_ydx\le R\int_\Omega\int_{\Gamma_R}k_2^2\psi^2+(k_2-k_1)^2\rho_1^2d\sigma_ydx.
\end{align}
We assume that for all $y\in\Gamma_R$, we have
$$
\int_\Omega\rho_1^2dx\le K,
$$
this can be ensured by taking $\rho_{I1}$ smooth enough. Hence,
\begin{align}
\nonumber
&\int_\Omega\int_{\Gamma_R}k_2^2\psi^2+(k_2-k_1)^2\rho_1^2d\sigma_ydx\le \bar{k}\int_\Omega\int_{\Gamma_R}\psi^2d\sigma_ydx+K\|k_2-k_1\|^2_{\L^2(\Omega)}\\
\le
\nonumber
&\e\|\nabla_y\psi\|^2_{\L^2(\Omega;\L^2(Y))}+c_\e\|\psi\|^2_{\L^2(\Omega;\L^2(Y))}+K\|k_2-k_1\|^2_{\L^2(\Gamma_R)}.
\end{align}
Taking all the estimates above into consideration, and compensating terms by selecting small $\e>0$, we finally obtain
\begin{align}
\nonumber
&\frac{d}{2dt}\|\psi\|^2_{\L^2(\Omega;\L^2(Y))}+D_1\|\nabla_y\psi\|^2_{\L^2(\Omega;\L^2(Y))}\\
\nonumber
&\le c\left(\|\psi\|^2_{\L^2(\Omega;\L^2(Y))}+|\delta D|+\|k_2-k_1\|^2_{\L^2(\Gamma_R)}\right)
\end{align}
Applying Gr\"{o}nwall's inequality leads to
\begin{align}
\nonumber
\|\psi(t)\|^2_{\L^2(\Omega;\L^2(Y))}\le C\left[\|\delta\rho_I\|^2_{\L^2(\Omega;\L^2(Y))}+|\delta D|+\|k_2-k_1\|^2_{\L^2(\Gamma_R)}\right]
\end{align}
and, by integration over $[0,T]$,
\begin{align}
\nonumber
\|\psi\|^2_{\L^2(0,T;\L^2(\Omega;\L^2(Y)))}\le CT\left[\|\delta\rho_I\|^2_{\L^2(\Omega;\L^2(Y))}+|\delta D|+\|k_2-k_1\|^2_{\L^2(\Gamma_R)}\right].
\end{align}
It also follows that
\begin{align}
\nonumber
\|\nabla_y\psi\|^2_{\L^2(0,T;\L^2(\Omega;\L^2(Y)))}\le CT^2\left[\|\delta\rho_I\|^2_{\L^2(\Omega;\L^2(Y))}+|\delta D|+\|k_2-k_1\|^2_{\L^2(\Gamma_R)}\right].
\end{align}
Further, by (\ref{phiestimate}) and Poincar\'{e}'s inequality, we have
\begin{align}
\nonumber
\|\varphi\|^2_{\L^2(0,T;H_0^1(\Omega)}\le C\left(\|\psi\|^2_{\L^2(0,T;L^2(\Omega;\L^2(Y)))}+|\delta A|\right).
\end{align}
Taking all the above estimates together, we obtain
\begin{align}
\nonumber
&\|\varphi\|^2_{\L^2(0,T;H_0^1(\Omega)}+\|\psi\|^2_{\L^2(0,T;L^2(\Omega;H^1(Y)))}\\
\nonumber
&\le C\left(|\delta A|+|\delta D|+\|\delta\rho_I\|^2_{\L^2(\Omega;\L^2(Y))}+\|k_2-k_1\|^2_{\L^2(\Gamma_R)} \right),
\end{align}
which concludes the proof.
\end{proof}

\section{Local stability for the inverse Robin problem}

In this section, we shall study the inverse problem of recovering the micro-macro Robin coefficient $k\in\L^2(\Gamma_R)$ from measurement on $\Gamma_N$; the Neumann part of the boundary. (Usuallly, one thinks of $\Gamma_R$ as the inaccessible part of $\partial Y$, while $\Gamma_N$ is the accessible part.)
Our discussion is influenced by the work \cite{jiang16}. An alternative way of working could be by following the abstract result in \cite{bourgeois13}.

Recall that we denote 
$$
\mathcal{K}=\{k\in\L^2(\Gamma_R):0<\underline{k}\le k(y)\le\bar{k}\mbox{ for } y\in\Gamma_R\}, 
$$
the set of admissible Robin coefficients. 
Denote by $k^*$ the true Robin coefficient of our problem and define the set $\mathcal{V}(k^*,a)$ as
\begin{equation}
\nonumber
\mathcal{V}(k^*,a)=\left\{k\in\mathcal{K}:\|k-k^*\|_{\L^2(\Gamma_R)}\le a\right\}.
\end{equation}  
Below $(\pi(k),\rho(k))$ denotes the solution to (\ref{generalProblem}) corresponding to the coefficient $k\in\mathcal{K}$. Our main result is the following theorem.
\begin{theorem}
\label{mainTeo}
Assume that $\rho_I\in H^1(\Omega, H^d(Y))$ and $\rho(k^*)\ge c_0>0$ on $[0,T]\times\Omega\times\Gamma_R$. Then there exists $a>0$ such that
\begin{equation}
\|\rho(k_2)-\rho(k_1)\|_{\L^2(0,T;\L^2(\Omega;\L^2(\Gamma_N)))}\ge c\|k_2-k_1\|_{\L^2(\Gamma_R)}
\end{equation}
for every $k_1,k_2\in\mathcal{V}(k^*,a)$.
\end{theorem}
\begin{remark}
The discussion around Theorem \ref{mainTeo} can be extended to the case of recovering micro-macro Robin coefficient with a genuine two-scale structure, e.g. $k\in\L^2(\Omega;\L^2(\Gamma_R))$ or $k\in\L^2(0,T;\L^2(\Omega;\L^2(\Gamma_R)))$. In this case, two-scale measurements are needed. To keep the presentation as simple as possible, we focus our attention on $k\in\mathcal{K}$.
\end{remark}
In the rest of this section, we prove establish several lemmata. The proof of Theorem \ref{mainTeo} is given in Section \ref{sectionMainTeo}.

\begin{lemma}
For any $k\in\mathcal{K}$ and $d\in\L^\infty(\Gamma_R)$ let $(\pi(k),\rho(k))$ be the solution to (\ref{generalProblem}) and $(u,v)=(u(k),v(k))$ the solution to
\begin{equation}
\label{frechet1}
\begin{cases}
-\Delta_xu=F(u,v) & \mbox{ in }\Omega\\
\partial_tv-D\Delta_yv = 0 & \mbox{ in }\Omega\times Y\\
-D\nabla_yv\cdot n_y+k(Rv-u)=d(\pi(k)-p_F-R\rho(k))&\mbox{ on } \Omega\times\Gamma_R\\
-D\nabla_yv\cdot n_y=0&\mbox{ on }\Omega\times\Gamma_N\\
u=0&\mbox{ at }\partial\Omega\\
v(0,x,y)=0& \mbox{ in }\Omega\times Y,
\end{cases}
\end{equation}
where $F(u,v)$ is specified below. Then $\rho(k)$ is continuously Fr\'{e}chet differentiable and its derivative $\rho'(k)d$ at $d\in\L^\infty(\Gamma_R)$ is given by $v(k)$.
\end{lemma}
\begin{proof}
One can observe that the well-posedness of (\ref{frechet1}) follows by similar arguments as in the previous sections. 
Take $k\in\mathcal{K}$ and $d\in\L^\infty(\Gamma_R)$ such that $k+d\in\mathcal{K}$.

Note first that $(u_1(k),v_1(k))=(\pi(k+d)-\pi(k),\rho(k+d)-\rho(k))$ solves
\begin{equation}
\nonumber
\begin{cases}
-\Delta_xu_1=f_1(u_1,v_1) & \mbox{ in }\Omega\\
\partial_tv_1-D\Delta_yv_1 = 0 & \mbox{ in }\Omega\times Y\\
-D\nabla_yv_1\cdot n_y+k(Rv_1-u_1)=d(\pi(k+d)-p_F-R\rho(k+d))&\mbox{ on } \Omega\times\Gamma_R\\
-D\nabla_yv_1\cdot n_y=0&\mbox{ on }\Omega\times\Gamma_N\\
u_1=0&\mbox{ at }\partial\Omega\\
v_1(0,x,y)=0&\mbox{ in }\Omega\times Y,
\end{cases}
\end{equation}
where $f_1(u_1,v_1)=f(\pi(k+d),\rho(k+d))-f(\pi(k),\rho(k))$. Denote by $F=f_1$ and
\begin{equation}
\nonumber
U=u_1-u=\pi(k+d)-\pi(k)-u,\quad V=v_1-v=\rho(k+d)-\rho(k)-v,
\end{equation}
then $(U,V)$ solves the problem
\begin{equation}
\nonumber
\begin{cases}
-\Delta_xU=f_2(U,V) & \mbox{ in }\Omega\\
\partial_tV-D\Delta_yV = 0 & \mbox{ in }\Omega\times Y\\
-D\nabla_yV\cdot n_y+k(RV-U)=d(u_1-Rv_1)&\mbox{ on } \Omega\times\Gamma_R\\
-D\nabla_yV\cdot n_y=0&\mbox{ on }\Omega\times\Gamma_N\\
U=0&\mbox{ at }\partial\Omega\\
V(0,x,y)=0&\mbox{ in }\Omega\times Y,
\end{cases}
\end{equation}
where $f_2(U,V)=f_1(u_1,v_1)-f_1(u,v)$.
Note that the nonlinearities $f_1,f_2$ satisfiy the conditions of the energy estimate Proposition \ref{energy1}. Thus,
\begin{align}
\nonumber
\|V\|^2_{\L^2(0,T;\L^2(\Omega,H^1(Y)))}\le C\|d\|^2_{\L^\infty(\Gamma_R)}\|u_1-Rv_1\|^2_{\L^2(\L^2(0,T;\L^2(\Omega;\L^2(\Gamma_R)))}.
\end{align}
Whence,
\begin{align}
\nonumber
&\frac{\|\rho(k+d)-\rho(k)-v\|_{\L^2(0,T;\L^2(\Omega,H^1(Y)))}}{\|d\|_{\L^\infty(\Gamma_R)}}\le\\
\nonumber
&\le C\left(\|u_1\|_{\L^2(\L^2(0,T;\L^2(\Omega;\L^2(\Gamma_R)))}+\|v_1\|_{\L^2(\L^2(0,T;\L^2(\Omega;\L^2(\Gamma_R)))}\right),
\end{align}
and it is sufficient to show that the right-hand side above tends to 0 as $\|d\|_{\L^\infty(\Gamma_R)}\rightarrow0$. Using the interpolation-trace inequality, we obtain
\begin{equation}
\nonumber
\|v_1\|_{\L(0,T;\L^2(\Omega;\L^2(\Gamma_R)))}\le C\|v_1\|_{\L^2(0,T;\L^2(\Omega,H^1(Y)))}.
\end{equation}
Furthermore, using Proposition \ref{energy1} and the interpolation-trace inequality again, we obtain
\begin{align}
\nonumber
&\|u_1\|^2_{\L^2(0,T;H^1(\Omega))}+\|v_1\|^2_{\L^2(0,T;\L^2(\Omega,H^1(Y)))}\le\\
\nonumber
&C\|d\|^2_{\L^\infty(\Gamma_R)}\|\pi(k+d)-p_F-R\rho(k+d)\|^2_{\L^2(0,T;\L^2(\Omega,\L^2(\Gamma_R)))}\le C'\|d\|^2_{\L^\infty(\Gamma_R)}
\end{align}
from which follows that
\begin{equation}
\nonumber
\lim_{d\rightarrow0}\frac{\|\rho(k+d)-\rho(k)-v\|_{\L^2(0,T;\L^2(\Omega,H^1(Y)))}}{\|d\|_{\L^\infty(\Gamma_R)}}=0.
\end{equation}
The proof of continuity follows by a similar argument, we refer to the discussion in \cite{jiang16}.
\end{proof}

We proceed now in a similar fashion as in e.g. \cite{choulli04,jiang16}. 

Let $g\in\L^2(\Gamma_R)$ and $(\theta,\omega)=(\theta(g),\omega(g))$ be the weak solution to the system
\begin{equation}
\label{inverse1}
\begin{cases}
-\Delta_x\theta=f(\theta,\omega) & \mbox{ in }\Omega\\
\partial_t\omega-D\Delta_y\omega = 0 & \mbox{ in }\Omega\times Y\\
-D\nabla_y\omega\cdot n_y+k^*R\omega=-g\rho(k^*)&\mbox{ on } \Omega\times\Gamma_R\\
-D\nabla_y\omega\cdot n_y=0&\mbox{ on }\Omega\times\Gamma_N\\
\theta=0&\mbox{ at }\partial\Omega\\
\omega(0,x,y)=0&\mbox{ in }\Omega\times Y
\end{cases}
\end{equation}
For $g\in\L^2(\Gamma_R)$, define the operator
\begin{equation}
\nonumber
N:\L^2(\Gamma_R)\rightarrow\L^2(0,T;\L^2(\Omega;\L^2(\Gamma_R)))
\end{equation}
by
\begin{equation}
\nonumber
N(g)=-D\nabla_y\omega(g)\cdot n_y.
\end{equation}
Then $N$ is a bounded linear operator (boundedness follow from energy estimates).
\begin{lemma}
The operator $N$ is bijective and $\|N^{-1}\|$ is finite.
\end{lemma}
\begin{proof}
We prove the surjectivity of $N$. Assume that $\varphi\in\L^2(0,T;\L^2(\Omega;\L^2(\Gamma_R)))$, we must prove that there exists $g\in\L^2(\Gamma_R)$ such that $N(g)=\varphi$. Using (\ref{inverse1}), we obtain
\begin{equation}
\nonumber
\varphi+k^*R\omega(g)=-g\rho(k^*),
\end{equation}
or, equivalently,
\begin{equation}
\label{N1}
\frac{\varphi}{\rho(k^*)}+g=-\frac{k^*R\omega(g)}{\rho(k^*)}.
\end{equation}
Define
\begin{equation}
\nonumber
\mathcal{O}:\L^2(\Gamma_R)\rightarrow\L^2(0,T;\L^2(\Omega;\L^2(\Gamma_R))).
\end{equation}
by
\begin{equation}
\nonumber
\mathcal{O}(g)=-\frac{k^*R\omega(g)}{\rho(k^*)}
\end{equation}
Then we have
\begin{equation}
\nonumber
\frac{\varphi}{\rho(k^*)}=(\mathcal{O}-I)(g).
\end{equation}
Note further that $\mathcal{O}=BA$, where
$$
A:g\mapsto\omega(g),\quad B:q\mapsto-\frac{k^*Rq}{\rho(k^*)}.
$$
We have seen that $A$ is compact and $B$ is clearly continuous.
Hence, $\mathcal{O}$ is compact. 

We claim now that $1$ is not an eigenvalue to $\mathcal{O}$. Then, by the Fredholm alternative theorem, $\mathcal{O}-I$ is invertible and
$$
g=(\mathcal{O}-I)^{-1}\left(\frac{\varphi}{\rho(k^*)}\right).
$$
To prove that 1 is not an eigenvalue of $\mathcal{O}$, assume that $\mathcal{O}(g)=g$ for some $g\in\L^2(\Gamma_R)$. It follows from (\ref{N1}) that
$\varphi/\rho(k^*)=0$, so $\varphi=N(g)=0$. Hence, $-D\nabla_y\omega(g)\cdot n_y=0$ on $[0,T]\times\Omega\times\Gamma_R$. Since $\omega(g)$ solves (\ref{inverse1}) it also solves
\begin{equation}
\label{inverse2}
\begin{cases}
-\Delta_x\theta=f(\theta,\omega) & \mbox{ in }\Omega\\
\partial_t\omega-D\Delta_y\omega = 0 & \mbox{ in }\Omega\times Y\\
-D\nabla_y\omega\cdot n_y=0&\mbox{ on }\Omega\times\partial Y\\
\theta=0&\mbox{ at }\partial\Omega\\
\omega(0,x,y)=0&\mbox{ in }\Omega\times Y.
\end{cases}
\end{equation}
Hence, $\omega(g)=0$, but then $-g\rho(k^*)=0$ from the Robin boundary condition of (\ref{inverse1}), and since $\rho(k^*)\ge c_0>0$, we get $g=0$. In other words, $1$ is not an eigenvalue of $\mathcal{O}$. In conclusion, $N$ is invertible. Since $N$ is bounded, bijective and linear, the open mapping theorem ensures that $N^{-1}$ exists and is bounded.
\end{proof}

\section{Proof of Theorem \ref{mainTeo}}
\label{sectionMainTeo}
We are now ready to prove our main result.
\begin{proof}[Proof of Theorem \ref{mainTeo}]
Let $\e>0$ and consider the scaled problem
\begin{equation}
\label{eq:scaledInverse1}
\begin{cases}
-\Delta_x\xi=f(\xi,\zeta) & \mbox{ in }\Omega\\
\partial_t\zeta-D\Delta_y\zeta = 0 & \mbox{ in }\Omega\times Y\\
-D\nabla_y\zeta\cdot n_y+kR\zeta=k(\xi+\e p_F)&\mbox{ on } \Omega\times\Gamma_R\\
-D\nabla_y\zeta\cdot n_y=0&\mbox{ on } \Omega\times\Gamma_N\\
\xi=0&\mbox{ at }\partial\Omega\\
\zeta(0,x,y)=\e\rho_I&\mbox{ in }\Omega\times Y.
\end{cases}
\end{equation}
Recall that we have $f(\e u,\e v)=\e f(u,v)$. From this it follows that the solution $(\xi^\e,\zeta^\e)$ to the above problem satisfies $(\xi^\e,\zeta^\e)=(\e\pi,\e\rho)$. Note that $\zeta^\e(k)=\e\zeta(k)\ge\e c_0>0$ on $[0,T]\times\Omega\times\Gamma_R$. 
Define the norm $\|\cdot\|_\e$ on $\L^2(\Gamma_R)$ by
\begin{equation}
\nonumber
\|g\|_\e=\left\|\frac{1}{\zeta^\e(k^*)}g\right\|_{\L^2(0,T;\L^2(\Omega,\L^2(\Gamma_R)))}.
\end{equation}

Further, define the mapping 
\begin{equation}
\nonumber
\sigma_\e:\mathcal{K}\rightarrow\L^2(0,T;\L^2(\Omega;\L^2(\Gamma_R))),
\end{equation}
by $\sigma_\e(k)=D\nabla_y(\zeta^\e)\cdot n_y$. It follows from the fact that $\zeta^\e(k)$ is Frechet differentiable with continuous derivative that $\sigma_\e$ is a $C^1$-diffeomorphism. We have
\begin{equation}
\label{sigmadef}
\sigma_\e'(k)g=D\nabla_y \zeta_1^\e(k^*,g)\cdot n_y
\end{equation}
where $\zeta_1^\e(k^*,g)$ is the solution to
\begin{equation}
\nonumber
\begin{cases}
-\Delta_x\xi_1^\e=f(\xi_1^\e,\zeta_1^\e(k^*,g)) & \mbox{ in }\Omega\\
\partial_t\zeta_1^\e(k^*,g)-D\Delta_y\zeta_1^\e(k^*,g)= 0 & \mbox{ in }\Omega\times Y\\
-D\nabla_y\zeta_1^\e(k^*,g)\cdot n_y+kR\zeta_1^\e(k^*,g)=-g\zeta^\e(k^*)&\mbox{ on } \Omega\times\Gamma_R\\
-D\nabla_y\zeta_1^\e(k^*,g)\cdot n_y=0&\mbox{ on } \Omega\times\Gamma_N\\
\xi_1^\e=0&\mbox{ at }\partial\Omega\\
\zeta_1^\e(k^*,g)(0,x,y)=0&\mbox{ in }\Omega\times Y.
\end{cases}
\end{equation}
Since $\zeta^\e(k^*)=\e\rho(k^*)$, it follows from (\ref{inverse1}) that $\zeta_1^\e(k^*,g)=\e\omega(g)$ and by $\sigma_\e'(k^*)g=\e N$ by (\ref{sigmadef}). It follows that $(\sigma'_\e(k^*)g)^{-1}=N^{-1}/\e$. Since $\sigma_\e'(k^*)g$ is a $C^1$-diffeomorphism, there exists a neighbourhood $N(k^*,a)$ such that for any $k_1,k_2\in N(k^*,a)$ we have
\begin{equation}
\nonumber
\|k_2-k_1\|_\e\le2\|(\sigma_\e(k^*)g^{-1})'\|\|\sigma'_\e(k_2)g-\sigma'_\e(k_1)g\|_{\L^2(0,T;\L^2(\Gamma_R))}
\end{equation}
(see the discussion in \cite{jiang16}).
We have
\begin{eqnarray}
\nonumber
\|k_2-k_1\|_\e&\le&2\|(\sigma_\e(k^*)g^{-1})'\|\|D\nabla_y\zeta^\e(k_2)\cdot n_y-D\nabla_y\zeta^\e(k_1)\cdot n_y\|_{\L^2(0,T;\L(\Omega;\L^2(\Gamma_R)))}\\
\nonumber
&\le&\frac{C}{\e}\|D\nabla_y\zeta^\e(k_2)\cdot n_y-D\nabla_y\zeta^\e(k_1)\cdot n_y\|_{\L^2(0,T;\L(\Omega;\L^2(\Gamma_R)))}.
\end{eqnarray}
Using (\ref{eq:scaledInverse1}), one obtains the estimate
\begin{align}
\nonumber
&\|D\nabla_y\zeta^\e(k_2)\cdot n_y-D\nabla_y\zeta^\e(k_1)\cdot n_y\|_{\L^2(0,T;\L(\Omega;\L^2(\Gamma_R)))}\\
\nonumber
&\le C\left(\|\zeta^\e(k_2)-\zeta^\e(k_1)\|_{\L^2(0,T;\L^2(\Omega;\L^2(\Gamma_R)))}+\|\xi^\e(k_2)-\xi^\e(k_1)\|_{\L^2(0,T;\L^2(\Omega))}\right)
\\
\nonumber
&+C\sum_{j=1}^2\|\zeta^\e(k_j)(k_2-k_1)\|_{\L^2(0,T;\L^2(\Omega;\L^2(\Gamma_R)))}.
\end{align}
By the interpolation-trace inequality, we have
\begin{equation}
\nonumber
\|\zeta^\e(k_2)-\zeta^\e(k_1)\|_{\L^2(0,T;\L^2(\Omega;\L^2(\Gamma_R)))}\le C\|\zeta^\e(k_2)-\zeta^\e(k_1)\|_{\L^2(0,T;\L^2(\Omega;H^1(Y)))}.
\end{equation}
Further, by the Poincar\'{e} inequality,
\begin{equation}
\nonumber
\|\xi^\e(k_2)-\xi^\e(k_1)\|_{\L^2(0,T;\L^2(\Omega))}\le C\|\xi^\e(k_2)-\xi^\e(k_1)\|_{\L^2(0,T;H^1(\Omega))} .
\end{equation}
Hence, we obtain
\begin{eqnarray}
\nonumber
\|k_2-k_1\|_\e&\le& \frac{C}{\e}\Bigg[\|\xi^\e(k_2)-\xi^\e(k_1)\|_{\L^2(0,T;H^1(\Omega))}+\|\zeta^\e(k_2)-\zeta^\e(k_1)\|_{\L^2(0,T;\L^2(\Omega;H^1(Y)))}\\
\nonumber
&+&\left.\sum_{j=1}^2\|\zeta^\e(k_j)(k_2-k_1)\|_{\L^2(0,T;\L^2(\Omega;\L^2(\Gamma_R)))}\right].
\end{eqnarray}
We have
\begin{equation}
\nonumber
\sum_{j=1}^2\|\zeta^\e(k_j)(k_2-k_1)\|_{\L^2(0,T;\L^2(\Omega;\L^2(\Gamma_R)))}\le C\|k_2-k_1\|_\e\sum_{j=1}^2\|\zeta^\e(k_j)\zeta^\e(k^*)\|_{\L^2(0,T;\L^2(\Omega;\L^\infty(\Gamma_R)))}
\end{equation}
Since $\Gamma_R$ is $(d-1)$-dimensional, we have
$$
\|\zeta^\e(k)\|_{\L^\infty(\Gamma_R)}\le c\|\zeta^\e(k)\|_{H^{d-1}(\Gamma_R)}\le c\|\zeta^\e(k)\|_{H^d(Y)}
$$
by Proposition \ref{trace} and Proposition \ref{sobolev}. Further, by Proposition \ref{parabolic} and the assumption $\rho_I\in \L^2(\Omega;H^{d-1}(Y))$, we have $\zeta^\e(k)\in \L^2(0,T;H^1(\Omega;H^d(Y)))$. Hence,
\begin{align}
\nonumber
\|\zeta^\e(k^*)\zeta^\e(k_1)\|_{\L^2(0,T;\L^2(\Omega;\L^\infty(\Gamma_R)))}&\le c\|\zeta^\e(k^*)\|_{\L^2(0,T;\L^2(\Omega;H^{d-1}(\Gamma_R)))}\|\zeta^\e(k_j)\|_{\L^2(0,T;\L^2(\Omega;H^{d-1}(\Gamma_R)))}\\
\nonumber
&\le c\|\zeta^\e(k^*)\|_{\L^2(0,T;\L^2(\Omega;H^d(Y)))}\|\zeta^\e(k_j)\|_{\L^2(0,T;\L^2(\Omega;H^d(Y)))}\\
\nonumber
&\le c\e^2\|\rho(k^*)\|_{\L^2(0,T;\L^2(\Omega;H^d(Y)))}\|\rho(k_j)\|_{\L^2(0,T;\L^2(\Omega;H^d(Y)))}\\
\nonumber
&\le C'\e^2.
\end{align}
By the above estimates, we can rely on
\begin{align}
\nonumber
(1-C\e)\|k_2-k_1\|_\e\le\frac{C}{\e}\left(\|\xi^\e(k_2)-\xi^\e(k_1)\|_{\L^2(0,T;H^1(\Omega))}+\|\zeta^\e(k_2)-\zeta^\e(k_1)\|_{\L^2(0,T;\L^2(\Omega;H^1(Y)))}\right). 
\end{align}
Using the equivalent norm on $\L^2(0,T;\L^2(\Omega;H^1(Y)))$ given by Proposition \ref{equivalenSob}, we obtain
\begin{eqnarray}
\nonumber
(1-C\e)\|k_2-k_1\|_\e&\le&\frac{C}{\e}\left(\|\xi^\e(k_2)-\xi^\e(k_1)\|_{\L^2(0,T;H^1(\Omega))}+\|\nabla_y(\zeta^\e(k_2)-\zeta^\e(k_1))\|_{\L^2(0,T;\L^2(\Omega;H^1(Y)))}\right)\\
\nonumber
&+&\frac{C}{\e}\|\zeta^\e(k_2)-\zeta^\e(k_1)\|_{\L^2(0,T;\L^2(\Omega,\L^2(\Gamma_N)))}.
\end{eqnarray}

Set $X^\e:=\xi^\e(k_2)-\xi^\e(k_1)$ and $Z^\e:=\zeta^\e(k_2)-\zeta^\e(k_1)$, then $(X^\e,Z^\e)$ solves
\begin{equation}
\nonumber
\begin{cases}
-\Delta_xX^\e=f_1(X^\e,Z^\e) & \mbox{ in }\Omega\\
\partial_tZ^\e-D\Delta_yZ^\e= 0 & \mbox{ in }\Omega\times Y\\
-D\nabla_yZ^\e\cdot n_y+k_1RZ^\e=(k_2-k_1)\zeta^\e(k_2)&\mbox{ on } \Omega\times\Gamma_R\\
-D\nabla_yZ^\e\cdot n_y=0&\mbox{ on } \Omega\times\Gamma_N\\
X^\e=0&\mbox{ at }\partial\Omega\\
Z^\e(0,x,y)=0&\mbox{ in }\Omega\times Y,
\end{cases}
\end{equation}
where $f_1(X^\e,Z^\e)=f(\xi^\e(k_2),\zeta^\e(k_2))-f(\xi^\e(k_1),\zeta^\e(k_1))$.

Using Proposition \ref{energy1} and similar estimates as above, we obtain 
\begin{align}
\nonumber
&\|\xi^\e(k_2)-\xi^\e(k_1)\|_{\L^2(0,T;H^1(\Omega))}+\|\zeta^\e(k_2)-\zeta^\e(k_1)\|_{\L^2(0,T;\L^2(\Omega;H^1(Y)))}\\
\nonumber
&\le C\|(k_2-k_1)\zeta^\e(k_2)\|_{\L^2(0,T;\L^2(\Omega;\L^2(\Gamma_R)))}\le C\e^2\|k_2-k_1\|_\e.
\end{align}
This finally yields the crucial estimate
\begin{align}
\nonumber
(1-C'\e)\|k_2-k_1\|_\e\le\frac{C}{\e}\|\zeta^\e(k_2)-\zeta^\e(k_1)\|_{\L^2(0,T;\L^2(\Omega,\L^2(\Gamma_N)))}.
\end{align}
Choose $\e^*>0$ such that $1-C'\e^*=1/2$ and use that $\zeta^\e=\e\rho$, then we obtain
$$
\|k_2-k_1\|_{\e^*}\le C\|\rho(k_2)-\rho(k_1)\|_{\L^2(0,T;\L^2(\Omega,\L^2(\Gamma_N)))}.
$$
Finally, since $\zeta^*(k)\ge\e c_0$, we obtain that
\begin{equation}
\nonumber
\|k_2-k_1\|_{\L^2(\Gamma_R)}\le C\|k_2-k_1\|_{\e^*}\le C\|\rho(k_2)-\rho(k_1)\|_{\L^2(0,T;\L^2(\Omega,\L^2(\Gamma_N)))}.
\end{equation}

\end{proof}

\bibliographystyle{plainnat}

\bibliography{literature}

\end{document}